\documentclass[a4paper,11pt,eqno]{article}
\usepackage[utf8]{inputenc}
\usepackage{amsmath}
\usepackage{amsfonts}
\usepackage{amssymb}
\usepackage{amsthm}
\usepackage{hyperref}
\usepackage{amsfonts}

\usepackage[all]{xy}
\usepackage{pdfpages}
\usepackage{fancyhdr}
\usepackage{makeidx}
\makeindex
\theoremstyle{plain}
\newtheorem{theorem}{Theorem}[section]
\newtheorem{prop}[theorem]{Proposition}
\newtheorem{lemm}[theorem]{Lemma}

\theoremstyle{definition}
\newtheorem{defin}[theorem]{Definition}
\newtheorem{exem}[theorem]{Example}
\newtheorem{rema}[theorem]{Remark}

 \renewcommand{\abstract}{\centerline{Abstract}}
 \date{}
 \title{\textbf{On a class of Algebras Satisfying polynomial identity of degree six}}

\author{Daouda KABRE and  Andr\'e CONSEIBO\\daoudakabre@yahoo.fr, andreconsebo@yahoo.fr\\D\'epartement de math\'ematiques\\Universit\'e Norbert ZONGO BP 376 Koudougou,
Burkina Faso }

\begin{document}
\pagestyle{empty}
%\includepdf[pages=-]{Pagedegarde7.pdf}
%\newpage
%\pagestyle{fancy}
%\thispagestyle{empty}
%\setcounter{page}{0}
%\hspace*{8cm}
%\cfoot{ii}
%\newpage
%\pagestyle{fancy}
%\thispagestyle{empty}
%\setcounter{page}{0}
\pagestyle{empty}
\thispagestyle{empty}
%\thispagestyle{empty}
%\setcounter{page}{0}
%\pagestyle{fancy}
%\cfoot{iv}
%\pagestyle{plain}
%\thispagestyle{empty}
%\setcounter{page}{0}
%\setcounter{page}{1}
%\renewcommand{\contentsname}{Sommaire}
%\tableofcontents
%\thispagestyle{empty}
%\setcounter{page}{0}
\thispagestyle{empty}
\setcounter{page}{0}

%\addcontentsline{toc}{chapter}{Résumé}
\thispagestyle{empty}
\setcounter{page}{0}
\thispagestyle{empty}
\setcounter{page}{0}

%\addcontentsline{toc}{chapter}{Abstract}
\pagestyle{plain}
\thispagestyle{empty}
\setcounter{page}{0}
%\setcounter{page}{1}
%\newpage
%\renewcommand{\contentsname}{Sommaire}
%\tableofcontents
\pagestyle{empty}
\addtocontents{toc}{\protect\thispagestyle{empty}\protect\pagestyle{empty}}
\thispagestyle{empty}
\setcounter{page}{0}
\pagestyle{plain}
\setcounter{page}{1}
\pagenumbering{arabic}
\maketitle

\begin{abstract}
In this paper we study the structure of a class of algebras satisfying a polynomial identity of degree 6. We show, assuming the existence of a non-zero idempotent, that if an algebra satisfies such an identity, it admits a Peirce decomposition related to this idempotent. We studied the algebraic structure and highlighted the connections of the algebras of this class with Bernstein algebras, train algebras, Jordan algebras and power associative algebras.\\
{\bf Keywords}: Peirce decomposition,  Bernstein algebra, Jordan algebra, Power associative algebra, train algebra,polynomial identity, idempotent.\\
\textbf{2020 Mathematics Subject Classification}: Primary 17D92, 17A05.
\end{abstract}

\section{Introduction}
The beginning of the study of Bernstein algebras goes back to $1923$ with the work of Serge Bernstein who gave a mathematical proof of the principle of stationarity of Hardy-Weinberg (\cite{Bernstein1942}). But it was P. Holgate who algebraically defined the objects currently known as Bernstein algebras in 1975 (\cite{AC9}). And since then, several authors have invaded this field of research through several publications (as examples, see \cite{Micali1989},\cite{AC1}, \cite{AC2}). The aim of this paper is to study a class of algebras verifying the polynomial identity $2x^2x^4=\omega(x)^2x^4+\omega(x)^4x^2$. This class of algebra, which contains the Bernstein algebra, models a population whose genetic crossing between the second generation and the fourth generation produces individuals with equal proportions of the genetic characters of both populations. We will first show that there exists an algebra verifying this identity but which is not a Bernstein algebra and prove that this algebra has an idempotent.
Assuming the existence of nonzero idempotent, we show that any algebra of this class admits a Peirce decomposition. The use of the Peirce decomposition will allow us to finally establish links between this class of algebras and well known algebras such as principal train algebras, Bernstein algebras, Jordan algebras and power associative algebras.

\section{Preliminaries}
Let $K$ be a commutative field and $A$ a commutative $K$-algebra, not necessarily associative.
For any element $x$ of $A$ we define the principal powers and the plenary powers of $x$ respectively by:\\
 $ x^1=x$, $x^{k+1}=xx^{k}$ and $ x^{[1]}=x$, $ x^{[k+1]}=x^{[k]}x^{[k]}$ for any integer $ k\geq 1$.

\begin{defin}
We will say that the algebra $A$ is:
\begin{enumerate}
  \item[i)]  a power associative if any monogenic subalgebra of $A$ is associative, that is, if $x^{i}x^{j}=x^{i+j}$ for all integers $i,j\geq 1$;
  \item[ii)] a Jordan algebra if $ x^2(yx) =(x^2y)x$, for all $x, y$ in $A$;
  \item[iii)] a  baric if there exists a non-zero morphism of algebras $\omega : A \rightarrow K$. The morphism $\omega$ is then called the weight function of the algebra $A$. The weight of an element $x$ of $A$ is the scalar $\omega(x)$.
\end{enumerate}
\end{defin}

 \begin{rema}
   Any Jordan algebra is a power associative algebra (\cite{AC3},\cite{AC4}).
 \end{rema}

\begin{defin}
 A baric $K$-algebra $(A,\omega)$ is a principal train algebra of rank  $n\geq 2$ if there are scalars $ \gamma_1,\dots, \gamma_{n-1}\in K$ such that $ x^{n}+ \gamma_1\omega(x)x^{n-1}+\dots+\gamma_{n-1}\omega(x)^{n-1}x= 0$, where the integer $n\geq 2$ is the smallest having this property.
\end{defin}

 \begin{defin}
 A baric $K$-algebra $(A,\omega)$ is a Bernstein algebra if $(x^{2})^{2}=\omega(x)^{2}x^{2}$ for any $x$ in $A$.
 \end{defin}

In the rest of the document, $K$ denotes an algebraically closed infinite commutative field with characteristic different from $2$.

In \cite{WB}, it is shown that if $A$ denotes a Bernstein algebra, then for any $x$ in $A$, $2x^ix^j=\omega(x)^jx^i+\omega(x)^jx^i$, $\forall i,j \geq 2$; in particular, for $ i=2 $ and $ j=4 $, $2x^{2}x^{4}=\omega(x)^{2}x^{4}+\omega(x)^{4}x^{2}, \forall x \in A$. In this paper, our attention will be focused on the structure of baric algebras satisfying the latter polynomial identity. We will show through the following example that there exists an algebra which verifies this polynomial identity but which is not a Bernstein algebra.

\begin{exem}
  Let ($A = < e_1, e_2, e_3>,\omega)$  be the commutative baric $K$-algebra whose multiplication table is given by:

  $e_1^2=e_1+e_3, e_2^2=e_3, e_{1}e_{2}=\frac{1}{2}e_{2}+e_{3}, e_{1}e_{3}=ze_{3}$ with $ z=\frac{-1-i\sqrt{23}}{4}$ and the other products being zero; $\omega: A\rightarrow K$, the algebras homomorphism that $\omega(e_{3})=\omega(e_2)=0$ and $\omega(e_1)=1$.

  Let us put $ x= \alpha e_1+\beta e_2+\eta e_3$. We have:\\
  $ x^{2}= \alpha^{2}e_{1}+\alpha\beta e_2+(\alpha^2+\beta^2+2\alpha\beta+2z\alpha\eta)e_3$,\\
   $x^3= \alpha^3 e_1+\alpha^2\beta e_2+ [(z+1)\alpha^{3}+(z+1)\alpha\beta^{2}+(2z+2)\alpha^{2}\beta+(2z^{2}+z)\alpha^{2}\eta]e_3$,\\
   $x^{4}=\alpha^{4} e_{1}+\alpha^{3}\beta e_{2}+ [(z^{2}+z+1)\alpha^{4}+(z^{2}+z+1)\alpha^{2}\beta^{2}+(2z^{2}+2z+2)\alpha^{3}\beta+(2z^{3}+z^{2}+z)\alpha^{3}\eta]e_{3} $, which implies that\\
   $ 2x^{2}x^{4}= 2\alpha^{6} e_{1}+2\alpha^{5}\beta e_{2}+ 2[(z^{3}+z^{2}+2z+1)\alpha^{6}+(z^{3}+z^{2}+2z+1)\alpha^{4}\beta^{2}+(2z^{3}+2z^{2}+4z+2)\alpha^{5}\beta+(2z^{4}+z^{3}+3z^{2})\alpha^{5}\eta]e_3$.

   Since $\omega(x)=\alpha $, therefore:\\
$\omega(x)^{2}x^{4}= \alpha^{2}x^{4}=\alpha^{6} e_{1}+\alpha^{5}\beta e_{2}+ [(z^{2}+z+1)\alpha^{6}+(z^{2}+z+1)\alpha^{4}\beta^2+(2z^2+2z+2)\alpha^5\beta+(2z^3+z^2+z)\alpha^5\eta]e_3$ and\\ $\omega(x)^{4}x^{2}= \alpha^{4}x^{2}= \alpha^{6}e_{1}+\alpha^{5}\beta e_{2}+(\alpha^6+\alpha^4\beta^2+2\alpha^5\beta+2z\alpha^5\eta)e_3$, so $\omega(x)^{2}x^{4}+\omega(x)^{4}x^{2}=2\alpha^{6} e_{1}+2\alpha^{5}\beta e_2+ [(z^2+z+2)\alpha^6+(z^2+z+2)\alpha^4\beta^2+(2z^2+2z+4)\alpha^5\beta+(2z^3+z^2+3z)\alpha^5\eta]e_{3}$.

We have $2z^3+2z^{2}+4z+2=z^2+z+2$, $4z^3+4z^2+8z+4=2z^2+2z+4$ and $4z^4+2z^3+6z^2=2z^{3}+z^2+3z$, thus $A$ satisfies the identity $2x^2x^4=\omega(x)^2x^4+\omega(x)^4x^2$.

   We have also $(x^2)^2=\alpha^4 e_1+\alpha^3\beta e_2+ [(2z+1)\alpha^4+( 2z+1)\alpha^2\beta^2+(2z+2)\alpha^3\beta+2z^2\alpha^3\eta]e_3$, so $ (x^{2})^{2}\neq \omega(x)^{2}x^{2} $ and $A$ is not Bernstein algebra. Therefore, $A$ is an algebra verifying a polynomial identity of degree $6$ which is not a Bernstein algebra. We show that the set of nonzero idempotents of $A$ is $\{e_1+ae_2+\frac{(1+a)^2}{1-2z}, a\in K\}$.
\end{exem}	

\section{Peirce decomposition}
\begin{lemm}\label{p1}
Let $(A,\omega)$ be an $K$-algebra verifying:
\begin{equation}\label{e1}
		2x^{2}x^{4}=\omega(x)^{2}x^{4}+\omega(x)^{4}x^{2}
\end{equation}
		For all $x, y$ in $A$ we have:\\
		\begin{enumerate}
			\item[\textbf{i)}] $4x^2(x(x(xy)))+2x^2(x(x^2y))+2x^2(x^3y)+4x^4(xy)=\omega(x)^2[2x(x(xy))+x(x^2y)+x^3y]+2\omega(xy)x^4+4\omega(x^3y)x^2+2\omega(x)^4(xy)$;
			\item[\textbf{ii)}] $x^2[4z(x(xy))+4x(z(xy))+4x(x(yz))+2z(x^2y)+4x(y(xz))+2y(x^2z)+4y(x(xz))]+ 4xz[2x(x(xy))+x(x^2y)+x^3y]+4xy[x^3z+x(x^2z)+2x(x(xz))]+4x^4(yz) = \omega(x)^2[2z(x(xy))+2x(z(xy))+2x(x(yz))+2y(x(xz))+2x(y(xz))+z(x^2y)+y(x^2z)]+ 2\omega(xz)[2x(x(xy))+x(x^2y)+x^3y]+2\omega(xy)[2x(x(xz))+x(x^2z)+x^3z]+12\omega(x^2yz)x^2+8\omega(x^3y)xz+8\omega(x^3z)xy+2\omega(x)^4yz+2\omega(yz) x^4$;
	\end{enumerate}
\end{lemm}
	
	\begin{proof}
	Theses identities are obtained by a partial linearization of identity (\ref{e1}).
	\end{proof}

The identities of the previous lemma allow us to establish that any algebra satisfying (\ref{e1}) and having a nonzero idempotent admits a Peirce decomposition.
	\begin{theorem}\label{t0}
		Let $(A,\omega)$ be a $K$-algebra verifying (\ref{e1}) and $e$ be a non-zero idempotent of $A$. Then $A$ admits a Peirce decomposition relative to $e$: $A= Ke\oplus A_0\oplus A_{\frac{1}{2}}\oplus A_{\lambda} \oplus A_{\bar{\lambda}}$  where $A_{\alpha}=\{x\in Ker\omega, ex = \alpha x\}$, with $\alpha\in \{0;\frac{1}{2};\lambda= \frac{-1-i\sqrt{23}}{4}; \bar{\lambda}= \frac{-1+i\sqrt{23}}{4} \}$.
	\end{theorem}

\begin{proof}
By considering the identity i) of the Lemma \ref{p1}and then setting $x=e$ and $y \in Ker\omega$ we obtain: $6e(ey)+2e(e(ey))+4e(e(e(ey)))=3ey+e(ey)+2e(e(ey))$, which implies that $4e(e(e(ey)))+5e(ey)-3ey=0$. By noting $\ell_e=L_e/ker(\omega)$ where $L_e : A \longrightarrow A, x \mapsto ex$, we have $4\ell_e^4+5\ell_e^2-3\ell_e=0$.
Thus $P(X)= 4X^{4}+5X^2-3X=4X(X-\frac{1}{2})(X-\lambda)(X-\bar{\lambda})$ (with $\lambda= \frac{-1-i\sqrt{23}}{4}$ and $\bar{\lambda}= \frac{-1+i\sqrt{23}}{4}$ ) is the minimal polynomial of $\ell_e$. According to the kernel lemma: $kerP(\ell_e)= ker\ell_e\oplus ker(\ell_e-\frac{1}{2}I)\oplus ker(\ell_e-\lambda I)\oplus ker(\ell_e-\bar{\lambda}I)$. By setting $A_{\alpha}=  ker(\ell_e-\alpha I)$, with $\alpha\in \{0;\frac{1}{2};\lambda= \frac{-1-i\sqrt{23}}{4}; \bar{\lambda}= \frac{-1+i\sqrt{23}}{4} \}$, we obtain the following decomposition: $A= Ke\oplus A_0\oplus A_{\frac{1}{2}}\oplus A_{\lambda} \oplus A_{\bar{\lambda}}$.
\end{proof}	
	
	\begin{theorem}\label{t1}
		Let $ A= K_{e}\oplus A_0\oplus A_{\frac{1}{2}}\oplus A_{\lambda} \oplus A_{\bar{\lambda}}$ be the Peirce decomposition of an algebra verifying \eqref{e1}, then:
		\begin{enumerate}
			\item[i)] $ A_{0}A_{0}\subset A_{\frac{1}{2}}$;
			\item[ii)] $ A_{\frac{1}{2}}A_{\frac{1}{2}}\subset A_0\oplus A_{\lambda} \oplus A_{\bar{\lambda}}$;
			\item[iii)] $ A_{\lambda} A_{\bar{\lambda}}={0}$;
			\item[iV)] $ A_{\lambda} A_{\lambda}={0}$;
			\item[$ \dot{V}) $] $ A_{\bar{\lambda}} A_{\bar{\lambda}}={0}$:
			\item[Vi)] $ A_{0}A_{\frac{1}{2}} \subset A_{\frac{1}{2}} \oplus A_{\lambda} \oplus A_{\bar{\lambda}}$;	
			\item[Vii)] $ A_{\lambda}A_{\frac{1}{2}} \subset A_{\frac{1}{2}} \oplus A_{0} \oplus A_{\bar{\lambda}}$;	
			\item[Viii)] $ A_{\bar{\lambda}}A_{\frac{1}{2}} \subset A_{\frac{1}{2}} \oplus A_{0} \oplus A_{\lambda}$;	
			\item[iX)] $ A_{0}A_{\lambda}\subset A_{\frac{1}{2}}$;
			\item[X)] $A_{0}A_{\bar{\lambda}}\subset A_{\frac{1}{2}}$.
		\end{enumerate}
	\end{theorem}
	
	\begin{proof}
	Indeed, for $x=e$ and $y, z \in ker\omega$ in relation ii) of Lemma \ref{p1}, we have:
	\begin{multline}
	8(ez)(e(e(ey)))+ 4e(z(e(ey)))+ 4e(e(z(ey)))+ 4e(e(e(yz)))+ 4(ez)(e(ey)))+2e(z(ey))\\
	+ 4e(e(y(ez)))+8(ez)(ey)+2e(y(ez))+ 4e(y(e(ez)))+4ey(e(ez))+8(ey)(e(e(ez)))+4e(yz)\\
	-2z(e(ey))-2e(z(ey))-2e(e(yz))-z(ey)-2e(y(ez))-y(ez)-2y(e(ez))-2yz=0
	\end{multline}
	For $y\in A_{\mu}$ and $z\in A_{\gamma}$ we have $ey=\mu y$ and $ez=\gamma z$, so relation (2) becomes:
	\begin{multline}
	4e(e(e(yz)))+ (4\mu +4\gamma-2)(e(e(yz)))+ (4\mu^{2}+4\gamma^{2}+4)e(yz)\\
	+(8\gamma \mu^{3}+8\gamma^{3}\mu+4\gamma^{2}\mu+4\gamma \mu^{2}+8\gamma \mu-2\mu^{2}-2\gamma^{2}-\mu-\gamma-2) (yz)=0
	\end{multline}, so
	\begin{multline}
	[4\ell_e^3+ (4\mu +4\gamma-2)\ell_e^2+(4\mu^{2}+4\gamma^{2}+4)\ell_e\\
	+(8\gamma \mu^3+8\gamma^3\mu+4\gamma^2\mu+4\gamma \mu^2+8\gamma \mu-2\mu^2-2\gamma^2-\mu-\gamma-2)I] (yz)=0
	\end{multline}
Let us examine this relation by discussing the values of $\gamma$ and $\mu$.
	$\underline{i)}$ $ \gamma=\mu=0 $\\
	In this case the relation becomes $ (2\ell_e^3-\ell_e^2+2\ell_e-I)(yz)=0$ which implies that $(\ell_e-\frac{1}{2}I)(\ell_e^2+I)(yz)=0$. Since $\ell_e^2+I$ is injective then $(\ell_e-\frac{1}{2}I)(yz)=0$, thus $yz \in A_{\frac{1}{2}}$ and i).\\
	$\underline{ii)}$ $\gamma=\mu=\frac{1}{2}$\\
	The relation gives $(2\ell_e^3+\ell_e^{2}+3\ell_e)(yz)=0$  and $0, \lambda, \bar{\lambda}$ being roots of $2X^3+X^2+3X$, thus  $yz \in A_0\oplus A_{\lambda}\oplus A_{\bar{\lambda}}$ thus ii).\\
$\underline{iii)} $ $ \gamma=\bar{\lambda} \quad and \quad \mu=\lambda $\\
	We have $(4\ell_e^3-4\ell_e^2-7\ell_e-16I)(yz)=0 $ and since $0,\frac{1}{2}, \lambda, \bar{\lambda}$ are not roots of $4X^3-4X^2-7X-16$ hence iii).\\
$ \underline{iv)} $ $ \gamma=\lambda \quad and \quad \mu=\lambda $\\
	Then $(4\ell_e^3- (4+2i\sqrt{23})\ell_e^2+ (-7+i\sqrt{23}) \ell_e+(26-2i\sqrt{23})I)(yz)=0$ and hence iV).\\
	$ \underline{v)} $ $ \gamma=\bar{\lambda} \quad and \quad \mu=\bar{\lambda }$\\
	 we have $ (4\ell_e^3+(-4+2i\sqrt{23})\ell_e^2- (7+i\sqrt{23})\ell_e+(26+2i\sqrt{23})I)(yz)=0$ and V) hold.\\
$\underline{vi)} $ $ \gamma=\frac{1}{2} \quad and \quad \mu=0$\\
	We have $ (4\ell_e^3+5\ell_e-3I)(yz)=0 $, which  implies that $4(\ell_e-\frac{1}{2}I)(\ell_e-\lambda I)(\ell_e-\bar{\lambda}I)(yz)=0$, hence Vi).\\
	$\underline{vii)}$ $\gamma=\lambda \quad and \quad \mu=\frac{1}{2}$\\
	We get $ (4\ell_e^3+(-1-i\sqrt{23})\ell_e^2+(\frac{-1+i\sqrt{23}}{2})\ell_e)(yz)=0$ and then Vii).\\
	$\underline{viii)}$ $ \gamma=\bar{\lambda} \quad and \quad \mu=\frac{1}{2}$\\
	 we have $(4\ell_e^3+(-1+i\sqrt{23})\ell_e^2+(\frac{-1-i\sqrt{23}}{2})\ell_e)(yz)=0$  and $0, \frac{1}{2}, \lambda $, hence $yz \in A_{\frac{1}{2}}\oplus A_0\oplus A_{\lambda}$, thus  Viii) is roved.\\
	$ \underline{ix)} $ $ \gamma=\lambda \quad and \quad \mu=0$\\
	 we have $(4\ell_e^3+(-3-i\sqrt{23})\ell_e^2+(\frac{-3+i\sqrt{23}}{2})\ell_e+I)(yz)=0 $, then $yz \in A_{\frac{1}{2}}$ and we have iX).\\
	$\underline{x)}$ $\gamma=\bar{\lambda} \quad and \quad \mu=0$\\
	 we have $ (4le^{3}+(-3+i\sqrt{23})le^{2}+(\frac{-3-i\sqrt{23}}{2})le+I)(yz)=0 $; et $\frac{1}{2} $ and we show that x) hold.
\end{proof}
	
	\begin{lemm}\label{l2}
		If $ A_{\lambda}=A_{\bar{\lambda}} =0$ for all $x_0\in A_0$ and $x_{\frac{1}{2}}\in A_{\frac{1}{2}}$, we have:
		\begin{enumerate}
			\item[i)] $ (x_{0}^{2})^{2}=0$;
			\item[ii)] $x_{\frac{1}{2}}^3=0$;
			\item[iii)] $( x_{\frac{1}{2}}^{2})^{2} =0$;
			\item[iV)]$x_{\frac{1}{2}}x_{0}^{2} =0$;
			\item[V)] $x_{\frac{1}{2}}(x_{\frac{1}{2}}x_{0})=0 $;
			\item[Vi)] $(x_{0}x_{\frac{1}{2}})^{2}=0$;
			\item[Vii)] $x_{\frac{1}{2}}^{2}(x_{0}x_{\frac{1}{2}})=0$.
		\end{enumerate}
	\end{lemm}

	\begin{proof}
Suppose that $A_{\lambda}= A_{\bar{\lambda}}=0$. According to Theorem \ref{t1}, we have $A_0^2\subset A_{\frac{1}{2}}$, $A_{\frac{1}{2}}^2\subset A_0$ and $A_0A_{\frac{1}{2}}\subset A_{\frac{1}{2}}$.\\
 Let $x=e+\alpha x_0$ be an element of weight $1$ of $A$ where $\alpha\in K$. The equality $2x^2x^4-x^2-x^4=0$ implies $\frac{\alpha^4}{2}(x_0^2)^2+2\alpha^5x_0^2x_0^3=0$, so $(x_0^2)^2=0$.\\
  Similarly, setting $x=e+\beta x_{\frac{1}{2}}$, we have $0=2x^2x^4-x^2-x^4=4\beta^3x_{\frac{1}{2}}^3+2\beta^4(x_{\frac{1}{2}}^4+(x_{\frac{1}{2}}^2)^2)$, hence the identities $x_{\frac{1}{2}}^3=0$ and $(x_{\frac{1}{2}}^{2})^{2} =0$ hold.\\
 Finally, by setting $x=e+\alpha x_0+\beta x_{\frac{1}{2}}$ be an element of weight $1$ of $A$ where $\alpha$ and $\beta$ are scalars. The equality $2x^2x^4-x^2-x^4=0$  implies $2\alpha^2\beta x_0^2x_{\frac{1}{2}}+6\alpha \beta^2 x_{\frac{1}{2}} (x_{0}x_{\frac{1}{2}})+ \alpha^3\beta[ x_{\frac{1}{2}}x_0^3+ 5x_0^2(x_{0}x_{\frac{1}{2}}) ]+ \alpha \beta^3[4x_{\frac{1}{2}}(x_{\frac{1}{2}}(x_{\frac{1}{2}}x_{0})))+2x_{\frac{1}{2}}(x_{0}x_{\frac{1}{2}}^{2})+8x_{\frac{1}{2}}^{2}(x_{0}x_{\frac{1}{2}})]+
	\alpha^2\beta^2[4x_{\frac{1}{2}}(x_0(x_0x_{\frac{1}{2}}))+x_{\frac{1}{2}}(x_{\frac{1}{2}}x_0^2)+8(x_{0}x_{\frac{1}{2}})^2+\frac{5}{2}x_0^2 x_{\frac{1}{2}}^2]=0$
By identifying the coefficients of $\alpha^i\beta^j$, $1\leq i+j\leq 4$ in the equality $2x^2x^4-x^2-x^4=0$, we have the seven identities.
\end{proof}
	
	\begin{lemm}\label{l3}
		If $A_{\lambda}=A_0=0$, for all $ x_{\bar{\lambda}}\in A_{\bar{\lambda}}$ and $x_{ \frac{1}{2}}\in A_{\frac{1}{2}} $, then:
		\begin{enumerate}
			\item[i)] $ x_{\frac{1}{2}}^3=0$;
			\item[ii)] $ x_{\bar{\lambda}}(x_{\frac{1}{2}}x_{\bar{\lambda}})=0$;
			\item[iii)] $ x_{\frac{1}{2}}(x_{\frac{1}{2}}x_{\bar{\lambda}})=0$;
			\item[iV)] $ x_{\frac{1}{2}}^{2}x_{\bar{\lambda}}=0$;
			\item[V)] $( x_{\frac{1}{2}}^{2})^2=0$;
			\item[Vi)] $(x_{\bar{\lambda}}x_{\frac{1}{2}})^2=0$;
			\item[Vii)] $x_{\frac{1}{2}}^{2}(x_{\bar{\lambda}}x_{\frac{1}{2}})=0$.
		\end{enumerate}
	\end{lemm}

	\begin{proof} It is similar to the proof of previous lemma.
	\end{proof}
	
Similarly, we also establish the following lemma.
	\begin{lemm}\label{l4}
		If $A_{\bar{\lambda}}= A_0=0$, for all $ x_{\lambda}\in A_{\lambda}$ and $x_{ \frac{1}{2}}\in A_{\frac{1}{2}} $, we have:
		\begin{enumerate}
			\item[i)] $x_{\frac{1}{2}}^3=0$;
			\item[ii)] $x_{\lambda}(x_{\frac{1}{2}}x_{\lambda})=0$;
			\item[iii)] $x_{\frac{1}{2}}(x_{\frac{1}{2}}x_{\lambda})=0$;
			\item[iV)] $x_{\frac{1}{2}}^2x_{\lambda}=0$;
			\item[V)] $(x_{\frac{1}{2}}^2)^2=0$;
			\item[Vi)] $(x_{\lambda}x_{\frac{1}{2}})^{2}=0$;
			\item[Vii)] $x_{\frac{1}{2}}^{2}(x_{\lambda}x_{\frac{1}{2}})=0$.
		\end{enumerate}
	\end{lemm}

\section{Link with Bernstein algebras}
The following result gives the necessary and sufficient conditions for an algebra verifying the identity $2x^2x^4=\omega(x)^{2}x^4+\omega(x)^4x^2$ to be a Bernstein algebra.
	\begin{theorem}\label{t2}
		Let $ A= Ke\oplus A_{0}\oplus A_{\frac{1}{2}}\oplus A_{\lambda} \oplus A_{\bar{\lambda}}$ be an algebra satisfying the identity $2x^{2}x^{4}=\omega(x)^{2}x^{4}+\omega(x)^{4}x^{2} $.  Then, $A$ is a Bernstein algebra if and only if $A_{\lambda}=A_{\bar{\lambda}}=0$.
	\end{theorem}
	
	\begin{proof}
Let $A= Ke\oplus A_0\oplus A_{\frac{1}{2}}\oplus A_{\lambda} \oplus A_{\bar{\lambda}}$ be an algebra verifying the identity $2x^2x^4=\omega(x)^2x^4+\omega(x)^4x^2$. Suppose that $A_{\lambda}=A_{\bar{\lambda}}=0 $. Let $x=e+x_0+x_{\frac{1}{2}}$ an element of weight $1$ in $A$.
 According to the lemma \ref{l2}, the quantities $ x_{\frac{1}{2}}^{3}, x_{\frac{1}{2}}x_{0}^{2}, x_{\frac{1}{2}}(x_{\frac{1}{2}}x_{0}),(x_{0}^{2})^{2}; (x_{\frac{1}{2}}^2)^2; x_{0}^{2}(x_{0}x_{\frac{1}{2}}), x_{\frac{1}{2}}^{2}(x_{0}x_{\frac{1}{2}})$ are zero. Therefore, we have $x^2=e+x_{\frac{1}{2}}+x_{\frac{1}{2}}^2+x_0^2+2x_0x_{\frac{1}{2}}$, $(x^2)^2=e+x_{\frac{1}{2}}+x_{\frac{1}{2}}^2+x_0^2+2x_0x_{\frac{1}{2}}$ and $(x^2)^2=x^2$. The set of elements of weight $1$ is dense in $A$ according to Zariski topology, thus $(x^2)^2=\omega(x)^2x^2$, $\forall x\in A$. Hence, $A$ is a Bernstein algebra. The reciprocal is obvious.
	\end{proof}
	
	Using the previous theorem and the characterization of Bernstein algebras which is Jordan algebras given in (\cite{Wa}) the following result.
	
	\begin{prop}
		Let $A=Ke\oplus A_0\oplus A_{\frac{1}{2}}\oplus A_{\lambda} \oplus A_{\bar{\lambda}}$ a algebra verifying the identity $2x^2x^4=\omega(x)^2x^4+\omega(x)^4x^2$. The following assertions are equivalent:
		\begin{enumerate}
\item[i)] $A$ is a Jordan algebra;
\item[ii)] $A$ is a power associative algebra;
\item[iii)] $A_{\lambda}=A_{\bar{\lambda}} =0$; $ A_0^{2}=0 $ and $x_0(x_0x_{\frac{1}{2}})=0$ for all $ x_{0}\in A_{0 }$ and $x_{ \frac{1}{2}}\in A_{\frac{1}{2}} $;
\item[iV)] $A$ is a principal train algebra of rank $3$ verifying the equation $x^{3}-\omega(x)x^2=0$. 	
		\end{enumerate}
	\end{prop}
	
\section{Relation with principal train algebras}
	\begin{prop}\label{p2}
		Let $ A= Ke\oplus A_{0}\oplus A_{\frac{1}{2}}\oplus A_{\lambda} \oplus A_{\bar{\lambda}}$ an algebra satisfying the identity $2x^{2}x^{4}=\omega(x)^{2}x^{4}+\omega(x)^{4}x^{2} $. If $ A_{\frac{1}{2}} =0$ then $A$ is a principal train algebra satisfying the equation $x^{5}-\frac{1}{2}\omega(x)x^{4}+\omega(x)^{2}x^{3}-\frac{3}{2}\omega(x)^{3}x^{2}=0$.
	\end{prop}
	
\begin{proof}
	$A_{\frac{1}{2}}$ being zero, we have $A_0^2=A_{\lambda}^2=A_{\bar{\lambda}}^2 =A_{\lambda}A_{0}=A_{\bar{\lambda}}A_{0}=A_{\bar{\lambda}}A_{\lambda}=0$.\\
For $x=e+x_0+x_{\lambda}+x_{\bar{\lambda}}$, we have $x^2=e+2\lambda x_{\lambda}+2\bar{\lambda} x_{\bar{\lambda}}$, $ x^3=e-3 x_{\lambda}-3 x_{\bar{\lambda}}$, $ x^{4}=e-2\lambda x_{\lambda}-2\bar{\lambda} x_{\bar{\lambda}}$, $ x^{5}=e+(2\lambda+3) x_{\lambda}+(2\bar{\lambda}+3) x_{\bar{\lambda}}$. So, $x^{5}-\frac{1}{2}x^{4}+x^{3}-\frac{3}{2}x^{2}=0$ and we obtain $x^5-\frac{1}{2}\omega(x)x^4+\omega(x)^2x^3-\frac{3}{2}\omega(x)^3x^2=0$ because the set of elements of weight $1$ is dense in $A$ according to Zariski's topology.
	\end{proof}
	
\begin{prop}\label{p3}
Let $A$ be an algebra satisfying the identity $2x^{2}x^{4}=\omega(x)^{2}x^{4}+\omega(x)^{4}x^{2 }$ Peirce decomposition $A= Ke\oplus A_0\oplus A_{\frac{1}{2}}\oplus A_{\lambda} \oplus A_{\bar{\lambda}}$ relative to an idempotent $e$.
 \begin{enumerate}
   \item[i)] If $A_0=A_{\bar{\lambda}}=0$, then $A$ satisfies the equation $x^3-(1+\lambda)\omega(x)x^2+\lambda\omega(x)^2x=0$;
   \item[ii)] If $A_0=A_{\lambda} =0$, then $A$ satisfies the equation $x^3-(1+\bar{\lambda})\omega(x)x^2+\bar{\lambda}\omega(x)^2x=0$.
 \end{enumerate}
\end{prop}

\begin{proof}
Let $A= Ke\oplus A_{0}\oplus A_{\frac{1}{2}}\oplus A_{\lambda} \oplus A_{\bar{\lambda}}$ be an algebra satisfying the identity $2x^{2}x^{4}=\omega(x)^{2}x^{4}+\omega(x)^{4}x^{2}$. Suppose $A_0=A_{\bar{\lambda}} =0$.\\
Let $ x=e+x_{\frac{1}{2}}+x_{\lambda} $ be an element of weight $1$ of $A$. By exploiting the relations of the lemma \ref{l4}, we have: $x^{2}=e+x_{\frac{1}{2}}+2\lambda x_{\lambda}+x_{\frac{ 1}{2}}^{2}+2x_{\frac{1}{2}}x_{\lambda}$, $ x^{2}-x= (2\lambda-1)x_{\lambda} +x_{\frac{1}{2}}^{2}+2x_{\frac{1}{2}}x_{\lambda}$,
$x(x^2-x)=\lambda x_{\frac{ 1}{2}}^2+(2\lambda^2-\lambda)x_{\lambda}+2\lambda x_{\lambda}x_{\frac{1}{2}}=\lambda(x^2-x)$, so $x^3-(1+\lambda)x^2+\lambda x=0$. Since the set of elements of weight $1$ is dense in $A$ by the Zariski's topology, then for any $x$ in $A$, we have $x^3-(1+\lambda)\omega(x)x^2+\lambda\omega(x)^2x=0$.
The proof of assertion ii) is done like that of the first assertion..
\end{proof}

	\begin{prop}\label{p4}
Let $A$ be an algebra satisfying the identity $2 x^{2}x^{4}=\omega(x)^{2}x^{4}+\omega(x)^{4}x^{2} $; then $A$ is a principal train algebra of rank $3$ if and only its train equation is of the form  $x^3-(1+\gamma)\omega(x)x^2+\gamma\omega(x)^2x =0$, where  $\gamma\in\{0, \lambda, \bar{\lambda}\} $	
	\end{prop}
	
\begin{proof}
	Let $A$ be an algebra satisfying the identity $2x^2x^4=\omega(x)^2x^4+\omega(x)^4x^2$.

Suppose $A$ is a principal train algebra of rank $3$, its equation is \begin{equation}\label{key}
x^3-(1+\alpha)\omega(x)x^2+\alpha\omega(x)^2x =0\quad with \quad \alpha \in K
\end{equation}
And a partial linearization of (\ref{key}) gives us \begin{equation}\label{keyli}
x^2y+2x(xy)-(1+\alpha)[\omega(y)x^{2}+2\omega(x)xy]+\alpha[2\omega(xy)x+\omega(x)^{2}y]=0
\end{equation}
setting $y=x^4$ in (\ref{keyli}), we have
\begin{equation}\label{key3}
x^2x^4+2x^6-(1+\alpha)\omega(x)^4x^2-2(1+\alpha)\omega(x) x^5+ 2\alpha\omega(x)^5x+\alpha\omega(x)^2x^4=0
\end{equation}
or $2x^6=2(1+\alpha)\omega(x)x^5-2\alpha\omega(x)^2x^4$, we also know that $x^2x^4=\frac{1}{2}\omega(x)^{2}x^{4}+\frac{1}{2}\omega(x)^{4 }x^{2}$; substituting $2x^6$ and $x^2x^4$ by their expressions in (\ref{key3}), we get
\begin{equation}\label{key4}
(\frac{1}{2}-\alpha)\omega(x)^{2}x^{4}-(\frac{1}{2}+\alpha)\omega(x)^{4} x^{2}+ 2\alpha\omega(x)^{5}x=0
\end{equation}
We can notice that $x^{3}=(1+\alpha)\omega(x)x^{2}-\alpha\omega(x)^{2}x $; which implies that $ x^{4}=(1+\alpha)\omega(x)x^{3}-\alpha\omega(x)^{2}x^{2}=(1+\alpha+\alpha^{2})\omega(x)^2x^2+(-\alpha^2-\alpha)\omega(x)^3x$. Substituting $x^{4}$ by its expression in (\ref{key4}), we get \\
$-\frac{1}{2}\alpha(2\alpha^{2}+\alpha+3)\omega(x)^{4}(x^{2}-\omega(x)x)= 0$. The algebra $A$ being of rank $3$ then $ x^{2}-\omega(x)x)\neq 0 $ which implies that $ -\frac{1}{2}\alpha(2\alpha^2+\alpha+3)=0$ hence $\alpha=0$, $\alpha=\lambda$ or $\alpha=\bar{\lambda}$.

Suppose $A$ is a principal train algebra of train equation  \begin{equation}\label{tr3}
x^3-(1+\alpha)\omega(x)x^2+\alpha\omega(x)^2x =0\quad with \quad \alpha \in \{0, \lambda, \bar{\lambda}\}
\end{equation}
If $\alpha=0$, $A$ is a Bernstein Jordan algebra (see \cite{Wa}) and therfore satisfies the identity $2x^2x^4=\omega(x)^2x^4+\omega(x)^4x^2$ (see \cite{WB}).

 For $\alpha=\lambda$, the  partial linearization of (\ref{tr3}) gives us \begin{equation}\label{keyL}
x^2y+2x(xy)-(1+\lambda)[\omega(y)x^{2}+2\omega(x)xy]+\lambda[2\omega(xy)x+\omega(x)^{2}y]=0
\end{equation}
 By setting $y=x^4$, we have $x^2x^4=-2x^6+(1+\lambda)[\omega(x)^4x^2+2\omega(x)x^5]-\lambda[2\omega(x)^5x+\omega(x)^2x^4]$, so
 $x^2x^4=[-2x^6+2(1+\lambda)\omega(x)x^5-2\lambda\omega(x)^2x^4]+\lambda\omega(x)^2x^4+ (1+\lambda)\omega(x)^4x^2-2\lambda\omega(x)^5x=(1+\lambda)\omega(x)^4x^2-2\lambda\omega(x)^5x+\lambda\omega(x)^2x^4$, hence

$x^2x^4-\frac{1}{2}\omega(x)^4x^2-\frac{1}{2}\omega(x)^2x^4=(\lambda^2+\frac{\lambda}{2}-\frac{1}{2})\omega(x)^3x^3+(-\lambda^2+\frac{3\lambda}{2}+\frac{1}{2})\omega(x)^4x^2-2\lambda
\omega(x)^5x$, so $x^2x^4-\frac{1}{2}\omega(x)^4x^2-\frac{1}{2}\omega(x)^2x^4=-2\omega(x)^3x^3+(2\lambda+2)\omega(x)^4x^2-2\lambda\omega(x)^5x$. Therefore $x^2x^4-\frac{1}{2}\omega(x)^4x^2-\frac{1}{2}\omega(x)^2x^4=-2\omega(x)^3(x^3-(\lambda+1)\omega(x)x^2+\lambda\omega(x)^2x)=0$ and $A$ satisfies the identity $2x^2x^4=\omega(x)^2x^4+\omega(x)^4x^2$. The proof is similar for $\alpha=\bar{\lambda}$.
	\end{proof}
	
	\begin{prop}\label{p5}
		Let $A=Ke\oplus A_0\oplus A_{\frac{1}{2}}\oplus A_{\lambda}\oplus A_{\bar{\lambda}}$ be an algebra satisfying the identity $2x^{2}x^{4}=\omega(x)^2x^4+\omega(x)^4x^2$; if $A$ is a principal train algebra of rank $4$, its train equation is one of the following forms:
		\begin{enumerate}
			\item[i)]  $ x^{4}-(1+\gamma)\omega(x)x^{3}+\gamma\omega(x)^{2}x^{2} =0$, $ \gamma \in \{0, \lambda, \bar{\lambda}\}$;
			\item[ii)] $ x^{4}-\frac{1}{2}\omega(x)x^{3}+\omega(x)^{2}x^{2}-\frac{3}{2}\omega(x)^{3}x =0 $;
			\item[iii)]  $ x^{4}-(\frac{3}{2}+\gamma)\omega(x)x^{3}+(\frac{1}{2}+\frac{3}{2}\gamma)\omega(x)^{2}x^{2}-\frac{1}{2}\gamma\omega(x)^{3}x =0$ ; $ \gamma \in \{\frac{1}{2}, \lambda, \bar{\lambda}\} $;
			\item[iV)]  $ x^{4}-(1+2\gamma)\omega(x)x^{3}+\gamma(\gamma+2)\omega(x)^{2}x^{2}-\gamma^{2}\omega(x)^{3}x =0$ ; $ \gamma \in \{\lambda, \bar{\lambda}\} $.
		\end{enumerate}
	\end{prop}

	\begin{proof}
	Let $A= Ke\oplus A_0\oplus A_{\frac{1}{2}}\oplus A_{\lambda} \oplus A_{\bar{\lambda}}$ be an algebra satisfying the identity $2x^2x^4=\omega(x)^2x^4+\omega(x)^4x^2$. Assuming $A$ a principal train algebra of rank $4$, its train equation is of the form
$ x^4-(1+\alpha+\beta)\omega(x)x^3+\alpha\omega(x)^2x^2+\beta\omega(x)^3x =0$ with $\alpha,\beta \in K$ so its minimal train polynomial is $P(X)=X(X-1)(X-\alpha_1)(X-\alpha_{2})= X^4-(1+\alpha_1+\alpha_2)X^3+(\alpha_1+\alpha_2+\alpha_1\alpha_2)X^2-\alpha_1\alpha_2X$; we then notice that $ \alpha=\alpha_{1}+\alpha_{2}+\alpha_{1}\alpha_{2} $ and $\beta=-\alpha_1\alpha_2$ therefore
\begin{equation}\label{key5}
x^{4}-(1+\alpha_{1}+\alpha_{2})\omega(x)x^{3}+(\alpha_{1}+\alpha_{2}+\alpha_{1} \alpha_{2})\omega(x)^{2}x^{2}-\alpha_{1}\alpha_{2}\omega(x)^{3}x =0
\end{equation}
Now let us look at the different cases related to the train roots $\alpha_{1}$ and $\alpha_{2}$:\\
	
\textbf{ $1st$ Case}: $\alpha_1\neq \alpha_2$, $ \alpha_1\neq \frac{1}{2}$ and $\alpha_{2}\neq \frac{1}{2}$\\
By exploiting the theorem $5$ of \cite{AC7} and the theorem \eqref{t0}; we observe that $A$ admits relatively to an idempotent $e$, the following Peirce decomposition: $A=Ke\oplus A_{\frac{1}{2}}\oplus A_{\alpha_{1}}\oplus A_{\alpha_{2}}=Ke\oplus A_0\oplus A_{\frac{1}{2}}\oplus A_{\lambda} \oplus A_{\bar{\lambda}}$ then we have by identification $\alpha_{1}, \alpha_2 \in \{0, \lambda, \bar{\lambda}\}$. Indeed:\\
for $\alpha_1=0$ and $ \alpha_{2}=\lambda $, (\ref{key5}) becomes $ x^{4}-(1+\lambda)\omega(x)x^{3}+\lambda\omega(x)^{2}x^{2} =0$, \\
for $\alpha_1=0$ and $ \alpha_2=\bar{\lambda}$, the equation  (\ref{key5}) becomes\\ $ x^{4}-(1+\bar{\lambda})\omega(x)x^{3}+\bar{\lambda}\omega(x)^{2}x^{ 2} =0$, and \\
if $\alpha_1=\lambda$ and $\alpha_2=\bar{\lambda}$, (\ref{key5}) becomes $x^4-\frac{1}{2}\omega(x)x^3+\omega(x)^2x^2-\frac{3}{2}\omega(x)^3x =0$ \\

\textbf{ $ 2nd$ Case}: $\alpha_{1}\neq \alpha_{2}$ and $\alpha_{1}= \frac{1}{2}$\\
Considering the theorem $1$ of \cite{AC8} and the theorem \eqref{t0}, it follows that $A$ admits the following Peirce decomposition: $A=Ke\oplus A_{\frac{1 }{2}}\oplus A_{\alpha_{2}}$ with $\alpha_{2}\in\{0, \lambda, \bar{\lambda}\}$. The train equation is therefore one of the following forms:\\
For $ \alpha_{1}=\frac{1}{2}$ and $ \alpha_{2}=0$, (\ref{key5}) becomes $ x^{4}-\frac{3}{2}\omega(x)x^{3}+\frac{1}{2}\omega(x)^{2}x^{ 2}=0$; \\
For $ \alpha_{1}=\frac{1}{2}$ and $ \alpha_{2}=\lambda$, the equation (\ref{key5}) becomes\\ $ x^{4}-(\frac{3}{2}+\lambda)\omega(x)x^{3}+(\frac{1}{2}+\frac{3} {2}\lambda)\omega(x)^{2}x^{2}-\frac{1}{2}\lambda\omega(x)^{3}x=0$; \\
For $ \alpha_{1}=\frac{1}{2}$ and $ \alpha_{2}=\bar{\lambda}$, (\ref{key5}) becomes\\ $x^{4}-(\frac{3}{2}+\bar{\lambda})\omega(x)x^{3}+(\frac{1}{2}+\frac{3}{2}\bar{\lambda})\omega(x)^{2}x^{2}-\frac{1}{2}\bar{\lambda}\omega(x)^{3}x=0$. \\

\textbf{ $ 3rd $Case}: $ \alpha_{1}= \alpha_{2} $, $\alpha_{1}\neq \frac{1}{2}$ and $ \alpha_{2}\neq \frac{1}{2}$\\
	According to the theorem $1$ of \cite{AC8} and as $A$ admits nonzero idempotents, the Peirce decomposition of $A$ with respect to an idempotent $e$ is\\ $A= Ke\oplus A_{\frac{ 1}{2}}\oplus B$ with $B=N\cap Ker(\ell_e-\alpha_{1}I)^{2}$. If $ B=0 $, we have $ x=e+x_{\frac{1}{2}} $ and $ x^2=e+x_{\frac{1}{2}}$ so $ x^2=\omega(x)x $ which is an elementary Bernstein algebra and this contradicts the fact that $A$ is a train algebra of rank $4$. Otherwise, there are three possibilities. Indeed:
	\begin{enumerate}
\item [i)] $\alpha_1=\alpha_2=0 $ implies that the train equation of $A$ is $x^4-\omega(x)x^3=0$;
	\item [ii)]$ \alpha_1=\alpha_2=\lambda$  implies that the train equation of $A$ is\\ $x^4-(1+2\lambda)\omega(x)x^3+\lambda(\lambda+2)\omega(x)^2x^2-\lambda^2\omega(x)^3x =0$;
	\item [iii)]$ \alpha_1=\alpha_2=\bar{\lambda}$ implies that the train equation of $A$ is\\ $x^4-(1+2\bar{\lambda})\omega(x)x^3+\bar{\lambda}(\bar{\lambda}+2)\omega(x)^2x^2-\bar{\lambda}^2\omega(x)^3x =0$.
\end{enumerate}
	\end{proof}
\begin{defin}
  For any fixed $\alpha$ in $K$, we consider the map\\ $\varphi_{\alpha}:K[X]\rightarrow K[X]$, $P\mapsto (X-\alpha)P$
\end{defin}

We easily establish the following lemma.

\begin{lemm}
  For $\alpha\in \mathbb{C}$, we have $\varphi_{\alpha}\circ\varphi_{\bar{\alpha}}=\varphi_{\bar{\alpha}}\circ\varphi_{\alpha}$
\end{lemm}

the proof of following Lemma is similar to Lemma\ref{l2}
\begin{lemm}\label{l7}
 If $A_0=0$, then $A_{\frac{1}{2}}^2\subset A_{\lambda}\oplus A_{\bar{\lambda}}$, $A_{\frac{1}{2}}A_{\lambda}\subset A_{\frac{1}{2}}\oplus A_{\bar{\lambda}}$, $A_{\frac{1}{2}}A_{\bar{\lambda}}\subset A_{\frac{1}{2}}\oplus A_{\lambda}$, $A_{\lambda}^2=A_{\bar{\lambda}}^2=A_{\lambda}A_{\bar{\lambda}}=0$ and
  for all  $x_{ \frac{1}{2}}\in A_{\frac{1}{2}}$; $x_{\lambda}\in A_{\lambda}$; $x_{\bar{\lambda}}\in A_{\bar{\lambda}}$ the following identities are verified:
		\begin{enumerate}
\item[i)] $[(\lambda+1)x_{ \frac{1}{2}}(x_{ \frac{1}{2}}^2)_{\lambda}+(\bar{\lambda}+1)x_{ \frac{1}{2}}(x_{ \frac{1}{2}}^2)_{\bar{\lambda}}]_{1/2}=0$;
\item[ii)] $(2\lambda+3)(x_{ \frac{1}{2}}(x_{ \frac{1}{2}}x_{\lambda})_{\bar{\lambda}})_{\lambda}+(\lambda+6)(x_{ \frac{1}{2}}(x_{ \frac{1}{2}}x_{\lambda})_{\frac{1}{2}})_{\lambda}=0$;
\item[iii)] $x_{\lambda}(x_{\lambda}x_{\frac{1}{2}})=0$;
\item[iv)]$ x_{\bar{\lambda}}(x_{\bar{\lambda}}x_{\frac{1}{2}})=0$;
\item[v)] $(x_{\lambda}(x_{\bar{\lambda}}x_{\frac{1}{2}}))_{\bar{\lambda}}=0$;
\item[vi)] $(x_{\bar{\lambda}}(x_{\lambda}x_{\frac{1}{2}}))_{\lambda}=0$;
\item[vii)] $(2\lambda-1)(x_{\lambda}(x_{\bar{\lambda}}x_{\frac{1}{2}}))+(2\bar{\lambda}-1)(x_{\bar{\lambda}}(x_{\lambda}x_{\frac{1}{2}}))=0$.
\end{enumerate}
\end{lemm}

\begin{theorem}
 \label{p6}
Let $A=Ke\oplus A_0\oplus A_{\frac{1}{2}}\oplus A_{\lambda}\oplus A_{\bar{\lambda}}$ be an algebra satisfying the identity $2x^2x^4=\omega(x)^2x^4+\omega(x)^4x^2$ such that $A_0=0$. Let setting $\mu=X^2-X$. If $A$ is principal train algebra of rank $n\geq 4$, its train equation is of the following form:
$\omega(x)^n((\varphi_{\bar{\lambda}}^t\circ \varphi_{\lambda}^s\circ\varphi_{1/2}^r)(\mu)(\frac{x}{\omega(x)})=0$, $r\geq 0, s\geq 0, t\geq 0$ are integers and $r+t+s=n-2$.
\end{theorem}

\begin{proof}
Let $x=e+x_{\frac{1}{2}}+x_{\lambda}+x_{\bar{\lambda}}$ an element of weight $1$ in $A$. We have $x^2-x=x_{\frac{1}{2}}^2+(2\lambda-1)x_{\lambda}+(2\bar{\lambda}-1)x_{\bar{\lambda}}+2x_{\frac{1}{2}}x_{\lambda}+2x_{\frac{1}{2}}x_{\bar{\lambda}}$. By setting $x^2-x=a_{\frac{1}{2}}+a_{\lambda}+a_{\bar{\lambda}}$ with $a_{\alpha}\in A_{\alpha}$, $\alpha\in \{\lambda, \bar{\lambda}\}$, we show using Lemma \ref{l7} that there exists an integer $r\geq 0$ such that $\varphi_{1/2}^r(\mu)(x)=a_{r,\lambda}+a_{r,\bar{\lambda}}$, $a_{r,\lambda}\in A_{\lambda}$, and $a_{r,\bar{\lambda}}\in A_{\bar{\lambda}}$. Similarly, there exists an integer $s\geq 0$ such that  $(\varphi_{\lambda}\circ\varphi_{1/2}^r)(\mu)(x)=b_{\bar{\lambda}}$ with $b_{\bar{\lambda}}\in A_{\bar{\lambda}}$.  Finally, for some integer $t\geq 0$, we have $(\varphi_{\bar{\lambda}}\circ\varphi_{\lambda}\circ\varphi_{1/2}^r)(\mu)(x)=0$. The set of element of weight $1$ being dense in $A$ according to Zariski topology, for any $x$ in $A$, we have $\omega(x)^n((\varphi_{\bar{\lambda}}^t\circ \varphi_{\lambda}^s\circ\varphi_{1/2}^r)(\mu)(\frac{x}{\omega(x)})=0$.
 \end{proof}

	\bibliographystyle{plain}

\end{document}